\documentclass[12pt]{article}
\usepackage{amsmath,amssymb,amsthm,amsfonts}

\usepackage{graphicx}
\newtheorem{thm}{Theorem}[section]
\newtheorem{cor}[thm]{Corollary}
\newtheorem{lem}[thm]{Lemma}
\newtheorem{prop}[thm]{Proposition}
\theoremstyle{definition}
\newtheorem{defn}[thm]{Definition}
\theoremstyle{remark}

\numberwithin{equation}{section}
\newcommand{\chglinespacing}[1]{\renewcommand{\baselinestretch}{#1}
\small\normalsize }

\newcommand{\chglinespacingmy}[1]{\renewcommand{\baselinestretch}{#1}
\small\normalsize }
\newcommand{\ftsection}[3][\normalsize]{{#1\section{#2}\label{#3}}}
\newcommand{\secsize}{\Huge}
\newcommand{\textdf}[1]{\textbf{#1}}

\newcommand{\SWALLOWCODE}[1]{}
\newcommand{\real}{{\mathbb R}}
\newcommand{\BM}{M} 
\newcommand{\ECP}{{\mathcal V}} 
\newcommand{\PD}{{\mathbb D}} 

\begin{document}

\title{Py-Calabi quasi-morphisms and quasi-states on orientable surfaces of higher genus \\
}

\renewcommand{\thefootnote}{\alph{footnote}}

\author{Maor Rosenberg}

\maketitle

\begin{abstract}
\noindent

We show that Py-Calabi quasi-morphism on the group of Hamiltonian
diffeomorphisms of surfaces of higher genus gives rise to
a quasi-state.

\end{abstract}


\chglinespacing{1.5}

\pagenumbering{roman}

\pagestyle{myheadings}
\chglinespacingmy{1.3}
\pagenumbering{arabic} 
\ftsection[\secsize]{Introduction}{sec:back}
In \cite{EP1} M. Entov and L. Polterovich establish an unexpected link between a group-theoretic notion of quasi-morphism,
which has been found useful in symplectic geometry, and a recently emerged branch of functional analysis called the theory of quasi-states and quasi-measures.
In this paper we show this connection for a recently discovered, due to P. Py \cite{Py}, Calabi quasi-morphism on orientable surfaces of higher genus.
The proof relies on hyperbolic geometry tools, surprisingly combined with combinatorial tools such as Hall's marriage theorem.

\subsection{The Group $\mbox{Ham}(\BM,\omega)$}
\begin{defn}\label{def:ham}
Let $\BM$ be a symplectic manifold equipped with a symplectic form $\omega$.
Let $F_t(x):=F(x,t)$, $F:\BM\times\real\rightarrow\real$ be a smooth function, called \textdf{Hamiltonian} function.
The pointwise linear equation $i_v w=-dF_t$ defines a vector field $v$ on $\BM$ denoted by $sgradF_t$.
The flow generated by the Hamiltonian vector field $sgradF_t$ is denoted by $f_t$.
By assuming that the union over $t$ of the supports of $F_t$ is contained in a compact subset of $\BM$,
we can guarantee that the above equation has a well defined solution for all $t$, and so $f_t$ is well defined.
The time-one map $f_1$, denoted by $\phi_F$, will be called \textdf{the Hamiltonian diffeomorphism generated by $F$}.
The collection of Hamiltonian diffeomorphisms has a group structure and this group is denoted by $\mbox{Ham}(\BM,\omega)$.
For further details see \cite{MS,P}.
\end{defn}

\subsection{Algebraic Results on $\mbox{Ham}(\BM,\omega)$}
The following algebraic results on $\mbox{Ham}(\BM,\omega)$ are due to Banyaga \cite{B}.

\begin{thm}\label{thm:banyaga1}
Let $M$ be a closed symplectic manifold, then $\mbox{Ham}(M,\omega)$ is simple, i.e., it has no non-trivial normal subgroup.
\end{thm}

\begin{thm}\label{thm:banyaga2}
Let $M$ be an open manifold with an exact symplectic structure, $\omega=d\lambda$.
Then $\mbox{Ham}(M,\omega)$ admits the Calabi homomorphism:
\[ Cal_M:\mbox{Ham}(M,\omega)\rightarrow\real \]
defined as
\[ Cal_M(\phi_F)=\int_0^1\int_M F(x,t){\omega}^m dt, \]
whose kernel is equal to the commutator subgroup of $\mbox{Ham}(M,\omega)$.
Furthermore, this kernel is a simple group.
\end{thm}

Note that $Cal_M$ is defined by $F$, but it can be shown that $Cal_M$ depends only on $\phi_F$ and not on the specific $F$.
Returning to the case where $\BM$ is closed, one cannot hope to construct a non-trivial homomorphism to $\real$ since $\mbox{Ham}(\BM,\omega)$ is simple.
However, for certain manifolds, one can find a map which is "locally" equal to the Calabi homomorphism and globally is a homomorphism up to a bounded error.
This map is called a Calabi quasi-morphism.
A precise definition is given in the following subsection.

\subsection{Calabi Quasi-morphism}\label{sec:CQM}

\begin{defn}
Let $G$ be a group, a function $\mu:G\rightarrow\real$ is a called a \textdf{quasi-morphism} if there exists a constant
$C$, called the \textdf{defect} of $\mu$, such that for each $x,y$ in $G$
\[ |\mu(xy)-\mu(x)-\mu(y)| < C .\]
If in addition $\mu(x^n)=n\mu(x)$ for each $x\in G$ and $n\in \mathbb{Z}$ then the quasi-morphism is called \textdf{homogeneous}.
Given a quasi-morphism we can define a homogeneous quasi-morphism called its \textdf{homogenization}
${\mu}_h$ by
\[ {\mu}_h(x) = \lim_{n\to\infty}\frac{\mu(x^n)}{n}. \]
\end{defn}
For further reading see e.g. \cite{K}.
\\
\par
Let $\BM$ be a closed manifold with a symplectic form $\omega$.
Let $U\subset\BM$ be open and connected. Denote by $\Gamma_U$ the subgroup of $\mbox{Ham}(\BM,\omega)$ generated by Hamiltonians supported in $U$.
If $\omega$ is exact on $U$ then, by Theorem \ref{thm:banyaga2}, $\Gamma_U$ admits the Calabi homomorphism $Cal_U$.
A set $U\subset\BM$ is called displaceable if there exists $f\in~\mbox{Ham}(\BM,\omega)$ such that $f(U)\cap\overline{U}=\emptyset$.
The following question was posed by M. Entov and L. Polterovich in \cite{EP}.
Can one construct a homogeneous quasi-morphism on $\mbox{Ham}(\BM,\omega)$ such that its restriction to $\Gamma_U$,
for any open, connected, exact and displaceable $U\subset\BM$, is equal to the Calabi homomorphism $Cal_U$?

\begin{defn}
A homogeneous quasi-morphism with the above property is called a \textdf{Calabi quasi-morphism}.
\end{defn}

M. Entov and L. Polterovich \cite{EP,Biran} have constructed Calabi quasi-morphisms for the case of the following symplectic manifolds:
$\mathbb{C}P^n$, a complex Grassmannian, $\mathbb{C}P^{n_1}\times ... \times\mathbb{C}P^{n_k}$ with a monotone product symplectic structure, the monotone symplectic blow-up of $\mathbb{C}P^2$ at one point.
Y. Ostrover extended it to some non-monotone manifolds \cite{Ost}.

The following result is due to P. Py \cite{Py}.

\begin{thm}\label{thm:pi1}
Let $\BM$ be an oriented closed surface of genus $g\ge 2$, equipped with a symplectic form $\omega$.
Then there exists a homogeneous quasi-morphism
\[ \mu:\mbox{Ham}(\BM,\omega)\rightarrow\real \]
such that the restriction to the subgroups $\Gamma_U$ is equal to the Calabi homomorphism, where $U$ is diffeomorphic to a disc or an annulus.
\end{thm}

In addition, P. Py has also constructed a Calabi quasi-morphism for the torus \cite{Py2}.

\begin{defn}
A smooth function $F:\BM\rightarrow\real$ is called a \textdf{Morse function} if all its critical points are non-degenerate.
If in addition the critical values of $F$ are distinct then $F$ is called a \textdf{generic Morse function}.
\end{defn}

\begin{defn}\label{def:commute}
Let $F:\BM\rightarrow\real$ be a generic Morse function.
Let ${\mathcal F}$ be the space of smooth functions on $\BM$ which commute with $F$ in the Poisson sense, i.e.
\[ {\mathcal F}:=\{H:\BM\rightarrow\real| \{F,H\}=0\}.\]
Note that $F$ and $H$ commute in the Poisson sense if and only if \[ \omega(sgradF,sgradH)=0.\]
Let $\Gamma$ be the set of time one maps corresponding to the flows generated by the Hamiltonian functions in ${\mathcal F}$, i.e.
\[ \Gamma:=\{\phi_H| H\in{\mathcal F}\}.\]
\end{defn}

Clearly, $\Gamma$ is an abelian subgroup of $\mbox{Ham}(\BM,\omega)$.
It is easy to show that if a homogeneous quasi-morphism is defined on an abelian group, then it is in fact a homomorphism.
Hence, the restriction of $\mu$, defined in Theorem \ref{thm:pi1}, on the subgroup $\Gamma$ is a homomorphism.
In \cite{Py} P. Py has proved the following formula for $\mu$ on $\Gamma$, assuming that the total area of $\BM$ is equal to $2g-2$,

\[ \mu(\phi_H)=\int_{\BM}H\omega - \sum_{x\in\ECP_F}H(x), \]
where $H\in{\mathcal F}$ and $\ECP_F$ is a certain subset of the critical points of $F$.
A precise formulation of this theorem will be given in Section \ref{sec:essential}.

\subsection{Quasi-state}
The notion of a quasi-state originates in quantom mechanics \cite{A,A1}, and has been a subject of intensive study in recent years following the paper \cite{A2} by J. F. Aarnes.
Here is the definition.

\begin{defn}
Denote by $C(\BM)$ the commutative (with respect to multiplication) Banach algebra of all continuous functions on $\BM$ endowed with the uniform norm.
For a function $F\in C(\BM)$ denote by ${\mathcal A}_F$ the uniform closure of the set of functions of the form $p\circ F$,
where $p$ is a real polynomial.
A (not necessarily linear) functional $\xi:C(\BM)\rightarrow\real$ is called a \textdf{quasi-state}, if it satisfies the following axioms:
\\
\textdf{Quasi-linearity.} $\xi$ is linear on ${\mathcal A}_F$ for every $F\in C(\BM)$.
\\
\textdf{Monotonicity.} $\xi(F)\le\xi(G)$ for $F\le G$.
\\
\textdf{Normalization.} $\xi(1)=1$.
\end{defn}
It is easy to show that a quasi-state is Lipschitz continuous with respect to the $C^0$-norm.
\\
\par
\textdf{Main Result.}
In the following, $\BM$ will be an oriented closed surface of genus $g\ge 2$, equipped with a symplectic form $\omega$,
and $\mu$ is Py's quasi-morphism given in Theorem \ref{thm:pi1}.
In \cite{EP1} M. Entov and L. Polterovich construct a quasi-state from a Calabi quasi-morphism,
Our goal is to show that this procedure is applicable to Py's Calabi quasi-morphism.
In the following, we assume that the total area of $\BM$, denoted by $Vol(\BM)$, is equal to $2g-2$.
The quasi-state is obtained from $\mu$ in the following way:
\\
\par
\begin{defn}\label{def:QS}
For a smooth function $F$ define

\[ \xi(F):=\frac{\int_{\BM}F\omega}{Vol(\BM)} - \frac{\mu(\phi_{F})}{Vol(\BM)} .\]
\end{defn}
The main result of the thesis is that the functional $\xi$ related to Py's quasi-morphism is a quasi-state.

\begin{thm} \label{thm:state_1}
The functional $\xi$ can be extended to $C(\BM)$, and the extension is a quasi-state.
\end{thm}

Note that this result implies that $\xi$ is Lipschitz continuous with respect to the $C^0$-norm.
\\
\par
\textdf{Organization of the work.}
In the following section we prove the main theorem assuming the monotonicity and continuity theorems, which are proved later on.
In sections \ref{sec:reeb}, \ref{sec:essential}, \ref{sec:pants}, \ref{sec:eight} we make the preparations for the monotonicity theorem,
which is proved in Section \ref{sec:monoton}.
In Section \ref{sec:reeb} we define the Reeb graph which is the base for the following constructions.
In Section \ref{sec:essential} we introduce the notion of essential critical points.
In Section \ref{sec:pants} we construct a pair of pants decomposition.
In Section \ref{sec:eight} we prove an intersection theorem of figure eights related to the pair of pants decomposition.
In Section \ref{sec:cont} we prove the continuity theorem by analyzing Py's construction of the quasi-morphism, this section can be read independently.

\ftsection[\secsize]{Main Steps}{sec:results}

The main ingredients of the proof are the following theorems.

\begin{thm} \label{thm:monotone_1}
Let $F,G:\BM\rightarrow\real$ be generic Morse functions,
such that $F\le G$.
Then $\xi(F)\le\xi(G)$.
\end{thm}

\begin{thm} \label{thm:cont_1}
The functional $\xi:C^{\infty}(\BM)\rightarrow\real$ is continuous with respect to the $C^2$-topology.
\end{thm}

We will prove Theorem \ref{thm:state_1} assuming Theorem \ref{thm:monotone_1} and \ref{thm:cont_1}.
\begin{proof}
Normalization is due to the fact that $\mu$ is a homogeneous quasi-morphism.
Indeed, $\phi_1$ corresponds to the identity element in the group $\mbox{Ham}(\BM,\omega)$,
so $\mu(\phi_1)=\mu(Id)=0$ by the homogeneity of $\mu$, and it follows that $\xi(1)=1$.
Since $\phi_{F+k}=\phi_F$ for any smooth function $F$ and a real constant $k$, we get from the definition of $\xi$ that
\begin{equation}\label{frml:line}
\xi(F+k)=\xi(F)+k .
\end{equation}
Let $\epsilon>0$ and $F$, $G$ be generic Morse functions, then
\[ |F-G|\le \|F-G\|_{C^0} ,\]
thus
\[ G- \|F-G\|_{C^0} \le F \le G+ \|F-G\|_{C^0} .\]
From the monotonicity of generic Morse functions (Theorem \ref{thm:monotone_1}) and Equation \ref{frml:line} we get
\[ \xi(G)- \|F-G\|_{C^0}\le \xi(F)\le \xi(G)+ \|F-G\|_{C^0} \]
so
\[ |\xi(F)-\xi(G)|\le  \|F-G\|_{C^0} \]
Thereby, $\xi$ is Lipschitz continuous on generic Morse functions with respect to the $C^0$-norm.
Generic Morse functions are $C^0$-dense in $C(\BM)$, therefore there is a unique extension of $\xi$ to a continuous map $\widehat{\xi}:C(\BM)\rightarrow\real$.
We will show that $\widehat{\xi}|_{C^\infty(\BM)}\equiv\xi$.
Indeed, for $H\in C^\infty(\BM)$ we can find a sequence of Morse functions $\{H_n\}$ that $C^2$-converges to $H$.
By Theorem \ref{thm:cont_1}, $\lim_{n\rightarrow\infty}\xi(H_n)=\xi(H)$.
In particular $\{H_n\}$ $C^0$-converges to $H$, so $\lim_{n\rightarrow\infty}\xi(H_n)=\widehat{\xi}(H)$ by definition.
Hence $\widehat{\xi}(H)=\xi(H)$, as required.
\\
\par
Monotonicity is easily extended to $\widehat{\xi}$ in the following way:
For $F$, $G\in~C(\BM)$ such that $F\le G$, choose generic Morse function sequences $\{F_n\}$, $\{G_n\}$ such that $\|F-F_n\|_{C^0}<\frac{1}{n}$
, $\|G-G_n\|_{C^0}<\frac{1}{n}$.
Define the sequences $\{F_n'\}$, $\{G_n'\}$ as follows: $F_n':=F_n-\frac{1}{n}$, $G_n':=G_n+\frac{1}{n}$.
Then for $n\in {\mathbb N}$,
\[ F_n'< F\le G<G_n'.\]
By the monotonicity on Morse functions we obtain $\xi(F_n')\le\xi(G_n')$ and by taking limits we get $\widehat{\xi}(F)\le \widehat{\xi}(G)$.
\\
\par
In order to show quasi-linearity we will first show a property called \textdf{strong quasi-additivity} which is defined as follows:\\
$\xi(F+G)=\xi(F)+\xi(G)$ for all smooth functions $F$, $G$ which commutes in the Poisson bracket, i.e. $\{F,G\}=0$.
The functional $\widehat{\xi}$ satisfies this property since it coincides with $\xi$ on smooth functions and the quasi-morphism $\mu$ is linear on commuting elements.
From the homogeneity of $\mu$ we get that $\xi$ is homogeneous and it is easily extended to $\widehat{\xi}$.
It is easy to see that strong quasi-additivity together with homogeneity yields quasi-linearity.
\end{proof}
\newpage
\ftsection[\secsize]{The Reeb Graph}{sec:reeb}

In this section we will define the Reeb graph \cite{Reeb} and prove a statement on its Euler characteristic.
The Reeb graph is a simple yet very useful tool in this work, and we will use it in the following sections to define the set of essential critical points,
and to construct the pair of pants decomposition.
\\
\begin{defn}
Let $\BM$ be a closed oriented surface of genus $g$.
Let $F:\BM\rightarrow\real$ be a generic Morse function.
Let $\{x_1, x_2,...,x_n\}\subset\BM$ be the set of critical points of $F$, with critical values $c_i=F(x_i)$, for $1\le i\le n$, such that $c_1<c_2<...<c_n$.
We will define the \textdf{Reeb graph of F}, $\Gamma(V,E)$, in the following way:\\
For each critical value $c_i$, the connected component of $F^{-1}(c_i)$ that contains $x_i$ can be:\\
1) The critical point $x_i$ in the case that its index is $0$ or $2$.\\
2) An immersed closed curve with a unique transversal double point $x_i$. This is the case when $x_i$ is of index $1$. \\
We will assign a vertex $v_i$ to the connected component of $F^{-1}(c_i)$ described above.
Let $C$ be the union of the above connected components, then $\BM\backslash C$ doesn't contain any critical points with respect to $F$.
Hence, by Morse theory \cite{Milnor}, it is a disjoint finite union of open cylinders.
The boundaries of each cylinder are contained in two connected components of $C$.
Define an edge associated with this cylinder between the two vertices that correspond to these components.
By Morse theory, each cylinder can be parameterized by $S^1\times(c_k,c_l)$ where $c_k$ and $c_l$ are the critical values of the critical points that bound the cylinder,
and each $S^1\times\{t\}$ is a connected component of the level curve $F^{-1}(t)$.
Parameterize the corresponding edge by the segment $(c_k,c_l)$.
Hence, we can define in a natural way a projection map $\pi_{\Gamma}:~\BM\rightarrow~\Gamma$,
by sending each connected component that contains a critical point $x_i$ to the vertex $v_i$,
and each connected component of the form  $S^1\times\{t\} \subset S^1\times(c_k,c_l)$ to the point $t$ in the corresponding edge, parameterized by $(c_k,c_l)$.

This is illustrated in Figure 1.

\begin{figure}[h]
\centering
\includegraphics{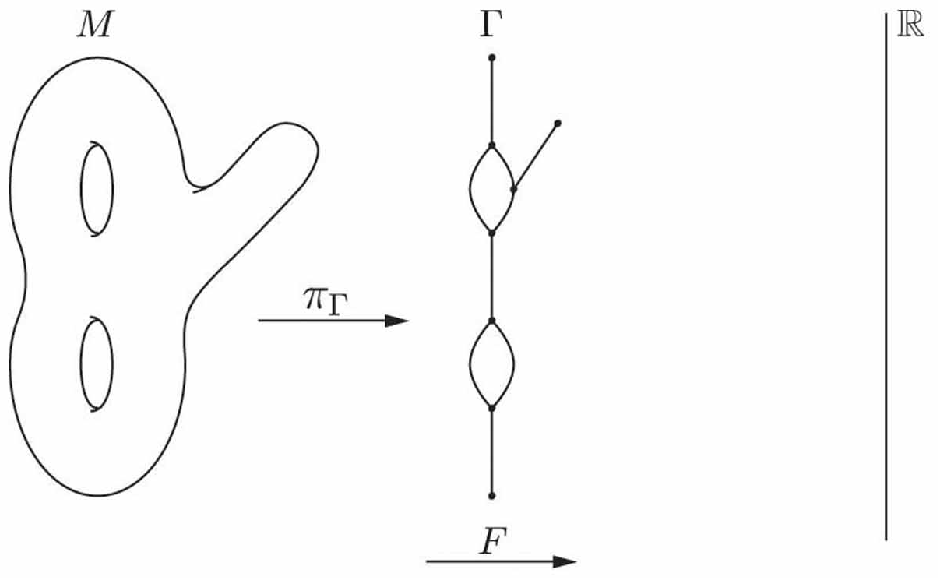}
\caption{$\Gamma$ is the Reeb graph of $F:\BM\rightarrow\real$. }
\label{fig:Reeb}
\end{figure}

Note that if $H:\BM\rightarrow\real$ is constant on each connected level curve of $F$ then we can define $H_{\Gamma}:\Gamma\rightarrow\real$ such
that $H=H_{\Gamma}\circ\pi_{\Gamma}$ by taking
\[
H_{\Gamma}(y) =
\left\{
\begin{array}{rl}
H(F^{-1}(c_i))  & y=v_i\\
H(F^{-1}(y))    & y \mbox{ in edge $(c_k,c_l)$.}
\end{array}
\right.
\]
\end{defn}

\begin{prop}\label{prop:graph_euler}
Let $\BM$ be a closed oriented surface of genus $g$.
Let $F:\BM\rightarrow\real$ be a generic Morse function.
Let $\Gamma$ be the Reeb graph of $F$.
Then the Euler characteristic $\chi(\Gamma)$ is equal to $1-g$.
\end{prop}

\begin{proof}
Let $\{x_1,...x_n\}$ be the set of critical points of $F$.
Recall the following formula for the calculation of the Euler characteristic
\[
	\chi(M)=\sum_{i=1}^n ind_{x_i}(grad F),
\]
where $ind_{x_i}(grad F)$ is the standard index of a critical point of the vector field $gradF$.
Observe that when $F$ is a Morse function, $ind_{x_k}(gradF)$ is $+1$ for local maximum and minimum points, and $-1$ for a saddle point.
Denote by $p$ the number of local maximum and minimum points, and by $q$ the number of saddle points.
Then by the above observation:
\[
	\chi(M)=p-q.
\]
From the definition of the Reeb graph, the number of vertices in $\Gamma$ is equal to the number of critical points in $\BM$, therefore:
\[
\#V = p+q.
\]
Furthermore, the degree of each vertex associated with a local maximum or minimum point is $1$,
and the degree of each vertex associated with a saddle point is $3$.
Hence, the number of edges in $\Gamma$ is
\[
	\# E=\frac{1}{2}\sum_{v\in V}deg(v) = \frac{1}{2}(p+3q).
\]
We can now calculate the Euler characteristic of $\Gamma$:
\begin{eqnarray*}
	\chi(\Gamma)=\# V - \# E &=& \\
= p+q-\frac{1}{2}(p+3q) = \frac{1}{2}(p-q) = \frac{1}{2}\chi(M) &=& \frac{1}{2}(2-2g)=1-g.
\end{eqnarray*}
\end{proof}

\ftsection[\secsize]{Essential Critical Points}{sec:essential}

The notion of essential critical points is needed for the precise formulation of Py's second theorem mentioned in Section \ref{sec:CQM}.
After defining the term and stating  Py's theorem,
we will prove a statement regarding the cardinality of the set of essential critical points.
The proof will clarify the concept, and its methods will also be used in Section \ref{sec:pants} for the construction of the pair of pants decomposition.
Similar methods have been used in \cite{Cole}.
\\
\begin{defn}
Let $\Gamma$ be a connected graph.
Given a vertex $v$, $\Gamma\backslash\{v\}$ is a disjoint union of connected topological spaces $Y_1\bigsqcup Y_2\bigsqcup\dots\bigsqcup Y_d$.
Let $\overline{Y_i}:=Y_i\bigcup\{v\}$ be the subgraph of $\Gamma$, which is composed of $Y_i$ attached to the vertex $v$.
We will refer to $\{\overline{Y_1},\dots ,\overline{Y_d}\}$ as the \textdf{subgraphs associated with $v$}.
\end{defn}

\begin{defn}
A vertex $v$ is called \textdf{non essential} if either one of the subgraphs associated with it, $\{\overline{Y_1},\dots ,\overline{Y_d}\}$, is a tree, or $v$ is an endpoint.
\end{defn}

\begin{defn}
Let $F$ be a generic Morse function defined on a closed surface of genus $g\ge 2$, and denote by $\Gamma_F$ the Reeb graph of $F$.
Define the \textdf{set of essential critical points of F}, $\ECP_F$, to be the critical points of $F$ that correspond to essential vertices in $\Gamma_F$.
\end{defn}

We can now state Py's second result \cite{Py}.
Let $F:\BM\rightarrow\real$ be a generic Morse function, where $\BM$ is a closed oriented surface of genus $g\ge 2$ and of total area $2g-2$.
Let ${\mathcal F}$ be the space of smooth functions on $\BM$ which commute with $F$ (Definition \ref{def:commute}),
and $\mu$ is Py's Calabi quasi-morphism given in Theorem \ref{thm:pi1}.

\begin{thm}\label{thm:pi2}
For $H$ in ${\mathcal F}$
\[ \mu(\phi_H)=\int_{\BM}H\omega - \sum_{x\in\ECP_F}H(x) \]
where $\ECP_F$ is the set of essential critical points of $F$.
\end{thm}

In the rest of the section we will show that the number of essential critical points is equal to $2g-2$, where $g\ge 2$ is the genus of $\BM$.

\begin{lem}\label{lem:atleast}
Let $\Gamma$ be a connected graph with vertices of degree $1$ or $3$, such that $\chi(\Gamma)\le -1$.
Assume that $\Gamma$ has at least one vertex of degree 1.
Then $\Gamma$ has more than two vertices.
\end{lem}
\begin{proof}
The assumption that $\Gamma$ has at least one vertex of degree $1$ implies that there are at least two vertices in $\Gamma$.
Assume that there are only two vertices in $\Gamma$ and denote them by $v_1$ and $v_2$.
Up to graph isomorphism, there are only two connected graphs with one vertex of degree $1$ and the other of degree $1$ or $3$.
One graph is simply $v_1,v_2$ and an edge connecting them, and the other has an extra edge connected on both ends to one of the vertices.
The Euler characteristic of the first graph is $1$, and of the second is $0$.
But we assume that $\chi(\Gamma)\le -1$, hence $\Gamma$ has more than two vertices.
\end{proof}

\textdf{Construction algorithm.}
Let $\Gamma$ be a connected graph with vertices of degree $1$ or $3$, such that $\chi(\Gamma)\le -1$.
Assume that there is at least one vertex of degree $1$.
We will define a new graph $\Gamma'$ obtained from $\Gamma$ in the following procedure.
Choose a vertex $v_1$ of degree $1$.
We will denote by $e$ the edge adjacent to $v_1$ and by $v_2$ the vertex on the other end of $e$.
The degree of $v_2$ is either $1$ or $3$. But if the degree is $1$,
it implies that $\Gamma$ has only two vertices contradicting Lemma \ref{lem:atleast}.
Therefore, the degree of $v_2$ is $3$.
Remove the vertex $v_1$ along with the edge $e$.
The degree of $v_2$ is now $2$.
Note that if $v_2$ is adjacent to both ends of the same edge,
it implies that $\Gamma$ has only two vertices contradicting Lemma \ref{lem:atleast}.
Therefore $v_2$ is adjacent to two different edges.
Remove the vertex $v_2$ and replace the two edges adjacent to it with one edge.
The new graph is not empty since $\Gamma$ has more than two vertices and we removed only two vertices.
Define $\Gamma'$ to be the new graph.

Note that $\Gamma '$ is not uniquely defined since the endpoint to be removed can be chosen arbitrarily.

\begin{lem}
$\Gamma '$ is a deformation retract of $\Gamma$.
\end{lem}

\begin{proof}
From the topological point of view, $\Gamma '$ is obtained from $\Gamma$ by contracting a line segment to a point.
Hence $\Gamma '$ is a deformation retract of $\Gamma$.
\end{proof}

\begin{cor}\label{cor:euler}
$\chi(\Gamma)=\chi(\Gamma')$ since the Euler characteristic is a topological invariant.
\end{cor}

\begin{lem}\label{lem:one_three}
The graph $\Gamma '$ is connected with vertices of degree $1$ or $3$.
\end{lem}

\begin{proof}
In the construction of $\Gamma'$, apart from the removed vertices $v_1$ and $v_2$,
the rest of the vertices have the same degree as in $\Gamma$.
Therefore the vertices in $\Gamma'$ are of degree $1$ or $3$.
\par
Let $v'$,$v''$ be any two vertices in $\Gamma'$.
Since $\Gamma$ is connected, there exists a path in $\Gamma$ between $v'$ and $v''$.
The vertex $v_1$ has degree $1$, so obviously the path can be chosen not to pass through $v_1$.
If the path passes through $v_2$ in $\Gamma$, then in $\Gamma'$ it will pass through the new edge that replaced $v_2$ and its two adjacent edges.
Hence $\Gamma'$ is also connected.
\end{proof}

\begin{lem}\label{lem:essential_iff}
Let $v\in\Gamma$ be an essential vertex, then $v\in\Gamma '$ and $v$ is essential in $\Gamma '$.
The opposite also holds, if $v\in\Gamma '$ is essential in $\Gamma '$ then $v$ is essential in $\Gamma$.
\end{lem}
\begin{proof}
Let $v\in \Gamma$ be an essential vertex.
The vertices that can be removed in the process of constructing $\Gamma '$ are either endpoints, or vertices that are connected via an edge to an endpoint.
The vertex $v$ is essential, hence it can not be an endpoint.
Furthermore, if $v$ is connected via an edge to an endpoint,
then there exists a subgraph associated with $v$ which is a tree, namely, it is the subgraph that contains the endpoint and $v$.
Hence, we get a contradiction to the fact that $v$ is essential.
We conclude that $v\in \Gamma '$.\\
Assume that $v$ is not essential in $\Gamma '$.
In the construction of $\Gamma '$, no new endpoints are created relative to those in $\Gamma$.
Now, $v$ is not an endpoint in $\Gamma$, hence it is not an endpoint in $\Gamma '$.
If one of the subgraphs associated with $v$ is a tree in $\Gamma '$, then it is also a tree in $\Gamma$,
since the addition of a free edge does not create a cycle.
But $v$ is essential in $\Gamma$, hence we have a contradiction, and $v$ is indeed essential in $\Gamma '$.\\
Conversely, let $v$ be essential in $\Gamma '$.
The vertices of $\Gamma '$ are contained in those of $\Gamma$, so obviously $v\in \Gamma$.
The vertex $v$ is not an endpoint in $\Gamma '$ so in particular it is not an endpoint in $\Gamma$.
The subgraphs associated with $v$ in $\Gamma '$ are all not trees, and the addition of a free edge does not change this property in $\Gamma$.
Hence $v$ is essential in $\Gamma$.
\end{proof}

\begin{figure}[h]
\centering
\includegraphics{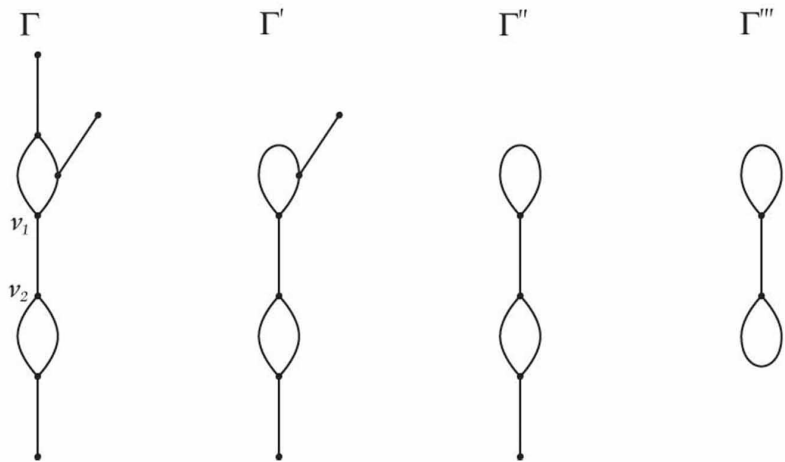}
\caption{The essential vertices of $\Gamma$ are $\{v_1,v_2\}$. $\Gamma'''$ has only vertices of degree 3, which are all essential.}
\label{fig:gamma}
\end{figure}

\begin{defn}\label{def:tilde_gamma}
By Corollary \ref{cor:euler} and Lemma \ref{lem:one_three} we can apply the above algorithm on $\Gamma$ recursively until there are no more vertices of degree $1$.
Denote the new graph by $\widetilde{\Gamma}$.
By Lemma \ref{lem:one_three}, $\widetilde{\Gamma}$ only has vertices of degree $3$ (see Figure \ref{fig:gamma}).
\end{defn}

\begin{prop}
Let $\Gamma$ be a connected graph such that all vertices are of degree $3$.
Then all vertices of $\Gamma$ are essential.
\end{prop}

\begin{proof}
Let $v\in\Gamma$.
Obviously $v$ can't be an endpoint since its degree is $3$.
Consider the subgraphs associated with $v$,  $\{\overline{Y_1},\dots ,\overline{Y_d}\}$.
Each subgraph has at most one vertex of degree $1$ (Namely, the vertex $v$).
But a non-trivial tree must have at least two vertices of degree $1$.
Therefore $v$ is essential.
\end{proof}

\begin{prop}\label{prop:num_essen}
Let $\Gamma$ be a connected graph with vertices of degree $1$ or $3$, such that $\chi(\Gamma)\le~-1$.
Then the number of essential vertices in $\Gamma$ is equal to $-2\chi(\Gamma)$.
\end{prop}

\begin{proof}
Using Corollary \ref{cor:euler} we get by induction $\chi(\Gamma)=\chi(\widetilde{\Gamma})$.
By Lemma \ref{lem:essential_iff}, we can see that $\Gamma$ and $\widetilde{\Gamma}$ have the same essential vertices.
Consequently, it is enough to prove the claim for $\widetilde{\Gamma}$.
Let $\#V$ and $\#E$ be the number of vertices and edges in $\widetilde{\Gamma}$, respectively.
Recall that the Euler characteristic of a graph is equal to $\#V-\#E$.
By Definition \ref{def:tilde_gamma}  $\widetilde{\Gamma}$ has only vertices of degree $3$.
Since every vertex is adjacent to three edges, and each edge is adjacent to two vertices, we get
\[ \#V=\frac{3}{2}\#E .\]
Thus
\[ \chi(\widetilde{\Gamma})=\#V-\#E=-\frac{1}{2}\#V \]
and
\[ \#V = -2\chi(\widetilde{\Gamma}) \]
as required.
\end{proof}

\begin{cor}
Let $\Gamma_F$ be the the Reeb graph of a generic Morse function $F$, defined on a closed surface $\BM$ of genus $g\ge 2$.
Then the number of essential vertices in $\Gamma_F$ is equal to $2g-2$.
\end{cor}

\begin{proof}
By Proposition \ref{prop:graph_euler} $\chi(\Gamma_F) = 1-g\le -1$.
Using Proposition \ref{prop:num_essen} the number of essential vertices in $\Gamma_F$ is equal to $-2\chi(\Gamma_F)$.

Therefore
\[ -2\chi(\Gamma_F) = -2(1-g) = 2g-2 \]
\end{proof}

\ftsection[\secsize]{The Pair of Pants Decomposition}{sec:pants}

The pair of pants decomposition is crucial for our proof of the monotonicity.
We will show here how to construct a pair of pants decomposition given a generic Morse function $F:\BM\rightarrow\real$.
\\
\par
Let $\ECP =\{x_1,x_2,...,x_{2g-2}\}$ be the set of essential critical points of $F$.
We will denote by $\{v_1,v_2,...,v_{2g-2}\}$ the set of essential vertices in $\Gamma$, the Reeb graph of $F$.
Let $F_{\Gamma}$ be the function on the graph $\Gamma$ induced by $F$.
Let $c_i=F(x_i)=F_{\Gamma}(v_i)$ for $i=1,...,2g-2$ be the critical values corresponding to the essential critical points.
Without loss of generality, we may assume that $c_1<c_2<...<c_{2g-2}$.
Choose small enough $\epsilon>0$ so that $[c_i-\epsilon,c_i+\epsilon]$ will only contain the critical value $c_i$.
\\
\begin{prop}
Denote by $(F_{\Gamma}^{-1}(c_i-\epsilon, c_i+\epsilon))_{v_i}$ the connected component of $F_{\Gamma}^{-1}(c_i-\epsilon, c_i+\epsilon)$ that contains the vertex $v_i$.
\\Then $\Gamma\backslash\bigcup_{i}(F_{\Gamma}^{-1}(c_i-\epsilon, c_i+\epsilon))_{v_i}$ is a disjoint union of trees, such that each tree has precisely two endpoints removed.
\end{prop}

\begin{proof}
Let $\widetilde{\Gamma}$ be the graph obtained from $\Gamma$ by applying the algorithm described above iteratively (Definition \ref{def:tilde_gamma}).
Recall that all vertices of $\widetilde{\Gamma}$ are essential and their number is $2g-2\ge 2$ for $g\ge 2$.
Hence $\widetilde{\Gamma}$ cannot contain only one vertex with no edges.\\
Note that  $\widetilde{\Gamma}\backslash\bigcup_{i}(F_{\widetilde{\Gamma}}^{-1}(c_i-\epsilon, c_i+\epsilon))_{v_i}$
is a collection of edges without the endpoints, which is of course also a collection of trees with exactly two endpoints removed.
We will use induction on the reverse steps of the algorithm to show that $\Gamma\backslash\bigcup_{i}(F_{\Gamma}^{-1}(c_i-\epsilon, c_i+\epsilon))_{v_i}$
is a collection of the required trees.
The base of the induction is the case of $\widetilde{\Gamma}$ which was shown above.
Denote by $\Gamma^{(n)}$ the graph obtained after the $n$-th iteration of the algorithm.
Assume that the claim holds for $\Gamma^{(n+1)}$ and we will show that it holds for $\Gamma^{(n)}$.
\\
$\Gamma^{(n+1)}\backslash\bigcup_{i}(F_{\Gamma^{(n+1)}}^{-1}(c_i-\epsilon, c_i+\epsilon))_{v_i}$
is a disjoint union of trees, such that each tree has precisely two endpoints removed.
In the reverse step of the algorithm we attach a free edge to one of the edges in the graph.
But an addition of a free edge to a tree is also a tree and there are still only two endpoints removed.
Hence $\Gamma^{(n)}\backslash\bigcup_{i}(F_{\Gamma^{(n)}}^{-1}(c_i-\epsilon, c_i+\epsilon))_{v_i}$
satisfies the inductive hypothesis.
Therefore the claim holds for $\Gamma=\Gamma^{(0)}$ as required.
\end{proof}

\begin{prop}\label{prop:cyl}
Denote by $(F^{-1}(c_i-\epsilon,c_i+\epsilon))_{x_i}$ the connected component of $F^{-1}(c_i-\epsilon,c_i+\epsilon)$ that contains $x_i$.\\
Then $\BM\backslash\bigcup_i(F^{-1}(c_i-\epsilon,c_i+\epsilon))_{x_i}$
is a disjoint union of cylinders with boundaries that corresponds to level sets of the form $F^{-1}(c_i\pm\epsilon)$ for $i=1,...,2g-2$.
\end{prop}

\begin{proof}
The connected components of $\BM\backslash\bigcup_i(F^{-1}(c_i-\epsilon,c_i+\epsilon))_{x_i}$
correspond to the connected components of $\Gamma\backslash\bigcup_{i}(F_{\Gamma}^{-1}(c_i-\epsilon, c_i+\epsilon))_{v_i}$
by the Reeb graph definition.
By the previous claim, each connected component of $\Gamma\backslash\bigcup_{i}(F_{\Gamma}^{-1}(c_i-\epsilon, c_i+\epsilon))_{v_i}$
is a tree with two end-points removed.
The analogue in $\BM\backslash\bigcup_i(F^{-1}(c_i-\epsilon,c_i+\epsilon))_{x_i}$
is a surface of genus zero with two boundary components corresponding to the level sets of the form $F^{-1}(c_i\pm\epsilon)$.
The values $c_i\pm\epsilon$ are regular, hence the boundary components have the structure of embedded circles.
By classification of surfaces these components are cylinders.
\end{proof}

\begin{defn}
By the term \textdf{pair of pants} we mean an embedding in $\BM$ of a connected orientable surface of genus zero with three boundary components.
\end{defn}

\textdf{Pair of Pants Decomposition.}
For $i=1,...,2g-2$, the critical point $x_i$ is essential, hence it is of index $1$.
Furthermore, it is the single critical point in $(F^{-1}(c_i-\epsilon,c_i+\epsilon))_{x_i}$.
Hence, by classification of surfaces, $(F^{-1}(c_i-~\epsilon,c_i+\epsilon))_{x_i}$
has the structure of an embedded surface of genus zero with three boundary components, or in other words, a pair of pants.
Denote by $l^i_1,l^i_2$ and $l^i_3$ the three boundary components of the pair of pants  $(F^{-1}(c_i-\epsilon,c_i+\epsilon))_{x_i}$.
By Proposition \ref{prop:cyl}, the complement of the union of all pairs of pants is a disjoint union of cylinders (Figure \ref{fig:decomp}, top).
The boundary components of each cylinder correspond to some two boundary components $l^i_t$ and $l^j_s$.
Hence, $l^i_t$ and $l^j_s$ are isotopic.
Note that $i$ can be equal to $j$ in the case that both boundaries of the cylinder belong to the same pair of pants.
There are $2g-2$ pairs of pants, each has $3$ boundary components. In total we have $6g-6$ boundary components.
Since each cylinder has $2$ boundary components, we get that there are $3g-3$ disjoint cylinders.
For each cylinder, choose one of its boundary components and attach the cylinder to a pair of pants along this boundary component,
denote the other boundary component of the cylinder by $\gamma_i$ for $i=1,...,3g-3$.
By attaching a cylinder to a pair of pants we again get a pair of pants.
As a result, $M\backslash\bigcup_i\gamma_i$ is a disjoint union of $2g-2$ pairs of pants.
We will denote by $P_i$ the pair of pants that contains the essential critical point $x_i$ (Figure \ref{fig:decomp}, bottom left).

\begin{figure}[tbp]
\centering
\includegraphics{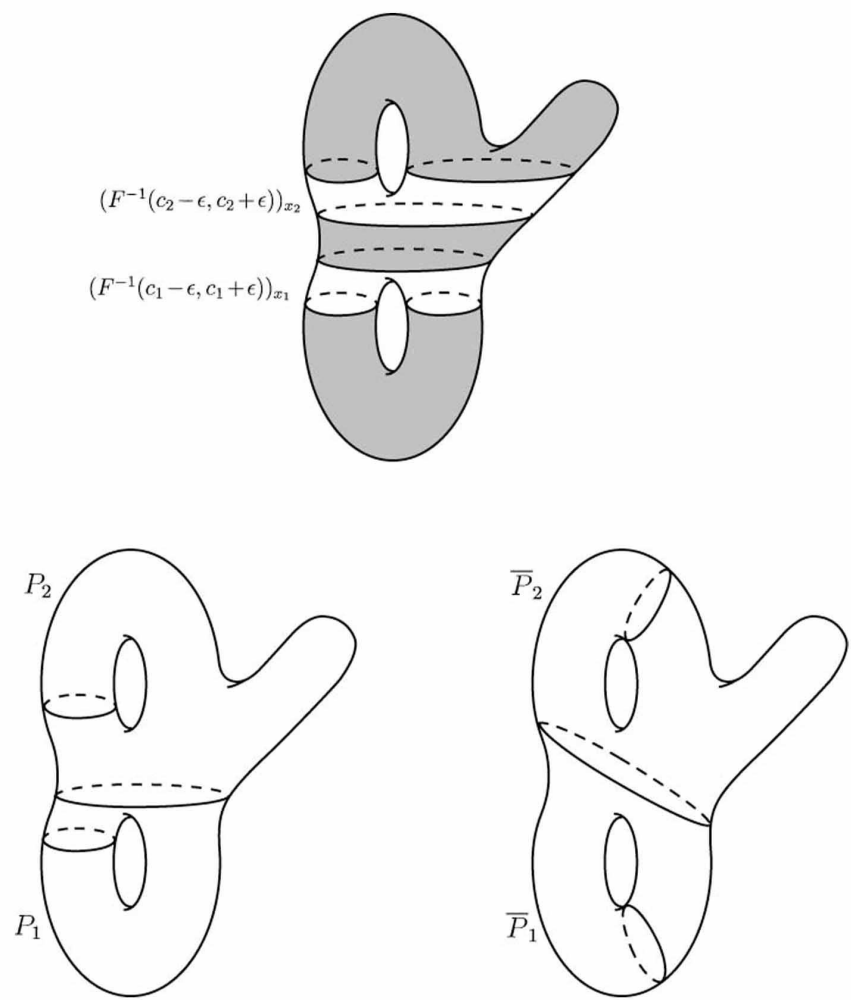}
\caption{The pair of pants decomposition.}
\label{fig:decomp}
\end{figure}

\begin{prop}
The circles in the collection $\{\gamma_1,...,\gamma_{3g-3}\}$ are disjoint, non-contractible, and pairwise non-isotopic.
\end{prop}

\begin{proof}
The circles are disjoint since they are boundaries of disjoint cylinders.
Let $\gamma\in\{\gamma_1,...,\gamma_{3g-3}\}$.
If $\gamma$ is contractible then it bounds a disc in $\BM$.
Let $e$ be the edge in the Reeb graph $\Gamma$, that contains the image of the level curve $\gamma$ by the natural projection $\pi:\BM\rightarrow\Gamma$,
i.e. $\pi(\gamma)\in e$.
Then $e$ is adjacent to an essential vertex $v$ on one end, and to a tree on the other end, corresponding to the disc.
But by the definition of an essential vertex, the subgraphs associated with $v$ are not trees, leading to a contradiction.
Therefore $\gamma$ is not contractible.\\
Let $\gamma_i,\gamma_j\in\{\gamma_1,...,\gamma_{3g-3}\}$ for $i\ne j$.
Assume that $\gamma_i$ is isotopic to $\gamma_j$.
Then $\gamma_i$ and $\gamma_j$ are the boundaries of some cylinder $C$ in $\BM$.
But $\gamma_i$ is also a boundary component of some pairs of pants $P^+,P^-\in\{P_1,...,P_{2g-2}\}$ (maybe a priory equal).
Thus, at least one of $P^+,P^-$ is contained in $C$.
This implies that at least one of the three boundary components of this pair of pants,
denoted by $\gamma_k\in\{\gamma_1,...,\gamma_{3g-3}\}$, is contractible, contradicting the first claim.
Hence, $\{\gamma_1,...,\gamma_{3g-3}\}$ are pairwise non-isotopic.
\end{proof}

Take an auxiliary metric of negative curvature on $\BM$.
By a theorem \cite{CB}, the submanifold that consists of the circles $\{\gamma_1,...,\gamma_{3g-3}\}$
is isotopic to a unique disjoint union of simple closed geodesics $\{\delta_1,...,\delta_{3g-3}\}$.
Note that $\BM\backslash \bigcup_i \delta_i$ is a disjoint union of pairs of pants since it is isotopic to $M\backslash\bigcup_i\gamma_i$.
We will denote by $\overline{P_i}$ the pair of pants with geodesic boundaries isotopic to $P_i$ (Figure \ref{fig:decomp}, bottom right).

\begin{defn}
We will call the collection $\{\overline{P}_1,...,\overline{P}_{2g-2}\}$, the \textdf{pair of pants decomposition of $M$ associated with $F$.}
\end{defn}

\begin{prop}
Given an auxiliary metric of negative curvature on $\BM$, the pair of pants decomposition of $\BM$ associated with $F$ is well-defined.
\end{prop}

\begin{proof}
The only choices we had in the definition, were the choice of a small enough $\epsilon>0$,
and which one of the boundary components of each cylinder will be denoted by $\gamma_i$.
Note that these choices do not affect the circles $\{\gamma_1,...,\gamma_{3g-3}\}$ up to isotopy.
The geodesics $\{\delta_1,...,\delta_{3g-3}\}$ are isotopic to  $\{\gamma_1,...,\gamma_{3g-3}\}$, and are uniquely determined, given the auxiliary metric.
Hence they do not depend on the above choices, and the term is well-defined.
\end{proof}

\ftsection[\secsize]{Figure Eight Intersections}{sec:eight}

We will define a figure eight collection of a generic Morse function $F:\BM\rightarrow\real$, and prove an intersection theorem,
using hyperbolic geometry tools and Hall's marriage theorem.
This is the last step towards the proof of monotonicity.
\\
\par
\begin{defn}
By the term \textdf{figure eight} we refer to an immersion in $\BM$ of a closed curve with a unique transversal double point.
\end{defn}

\begin{defn}
Let $\BM$ be a closed surface of genus $g\ge 2$ and let $F:\BM\rightarrow\real$ be a generic Morse function.
Let $\ECP$ be the set of essential critical points of $F$.
For each $x_i\in\ECP$, denote by $c_i=F(x_i)$ its critical value for $i=1,...,2g-2$.
We can assume $c_1<c_2<...<c_{2g-2}$.
Denote by $e_i$ the connected component of $F^{-1}(c_i)$ that contains $x_i$.
Note that $x_i$ is the only critical point in this level set, and its index is $1$ since it is essential.
By classification theory, $e_i$ is an immersed closed curve with a unique transversal double point at $x_i$.
Thereby, $e_i$ is a figure eight.
We will call the collection $\{e_1,...,e_{2g-2}\}$ the \textdf{figure eight collection of $F$}.
\end{defn}

We will use the following preliminary results from hyperbolic geometry. Proofs can be found in \cite{Buser}.

\begin{thm}\label{thm:prelim}
Let $S$ be a compact hyperbolic surface and let $\gamma$ be a non-contractible closed curve on $S$.
We will denote by $\PD$ the Poincar\'e disc model,
and by $\widehat{S}\subset\PD$ the universal cover of $S$.
Note that $\widehat{S}$ is isometric to $\PD$ when $S$ is without boundary.
Then the following hold:\\
(i) $\gamma$ is freely homotopic to a unique closed geodesic $\eta$.\\
(ii) For any lift $\widehat{\eta}$ of $\eta$ in the universal covering $\widehat{S}\subset\PD$ there exists a lift $\widehat{\gamma}$ of $\gamma$ such that  $\widehat{\gamma}$ and $\widehat{\eta}$ have the same endpoints at the circle at infinity.\\
(iii) $\eta$ is either contained in $\partial S$ or $\eta\cap\partial S=\emptyset$.
\end{thm}

\begin{thm}\label{thm:trans}
Let $\gamma_1,\gamma_2$ be two non-contractible closed curves,
and let $\eta_1$, $\eta_2$ be the unique closed geodesic curves freely homotopic to $\gamma_1$ and $\gamma_2$ respectively.
Then if $\eta_1$ and $\eta_2$ have transversal intersection, it implies $\gamma_1\cap \gamma_2\ne\emptyset$.
\end{thm}

\begin{lem}\label{lem:contained}
Let $x\in\ECP=\{x_1,...,x_{2g-2}\}$ be an essential critical point with critical value $c=F(x)$.
Let $\overline{P}\in\{\overline{P}_1,...,\overline{P}_{2g-2}\}$ be the corresponding geodesic pair of pants and $e\in\{e_1,...,e_{2g-2}\}$ the corresponding figure eight.
Then the unique closed geodesic homotopic to $e$, denoted by $\eta$, is contained in $\overline{P}$,
and $\overline{P}\backslash\eta$ is a disjoint union of three cylinders.
\end{lem}

\begin{proof}
The figure eight $e$ is defined as the connected component of the level set $F^{-1}(c)$ that contains $x$.
Note that $e$ is contained in the pair of pants $(F^{-1}(c-\epsilon,c+\epsilon))_x$, and $(F^{-1}(c-\epsilon,c+\epsilon))_x\backslash e$ is a disjoint union of three cylinders.
In the definition of the pair of pants decomposition, we construct an intermediate pair of pants $P$ by attaching cylinders to some (or none) of the boundary components of $(F^{-1}(c-\epsilon,c+\epsilon))_x$.
Following the notation used in the definition, we denote the new boundary components of the pair of pants by $\gamma_1$,$\gamma_2$ and $\gamma_3$.
Note that $\gamma_2$ and $\gamma_3$ coincide in the case in which two boundary components of the initial pair of pants are isotopic.
We only attached cylinders to boundaries of $(F^{-1}(c-\epsilon,c+\epsilon))_x$,
therefore $P\backslash e$ is also a disjoint union of three cylinders.
\par
By a theorem \cite{CB} there exists a homeomorphism $h:\BM\rightarrow\BM$ isotopic to the identity such that $\gamma_1$,$\gamma_2$ and $\gamma_3$ are sent to the unique closed geodesics $\delta_1$, $\delta_2$ and $\delta_3$ respectively.
The image of $P$ by $h$ is therefore $\overline{P}$, the pair of pants bounded by $\delta_1$, $\delta_2$ and $\delta_3$.
The image of $e$ by $h$, denoted by $\widetilde{e}:=h(e)$, is a figure eight homotopic to $e$
and contained in $\overline{P}$.
Furthermore, $h(P\backslash e)=\overline{P}\backslash \widetilde{e}$ is again a disjoint union of three cylinders since $h$ is a homeomorphism.
We can consider $\overline{P}\cup\partial\overline{P}$ as a compact hyperbolic surface with boundary,
and $\widetilde{e}$ is a non-contractible closed curve on $\overline{P}$.
Hence, by Theorem \ref{thm:prelim}, $\widetilde{e}$ is freely homotopic to a unique closed geodesic $\eta$,
and $\eta$ is either contained in $\partial\overline{P}$ or $\eta\cap\overline{P}=\emptyset$.
Since $\widetilde{e}$ is not homotopic to any of the boundary components of $\overline{P}$, $\eta$ cannot be contained in $\partial\overline{P}$.
Therefore, $\eta\cap\partial\overline{P}=\emptyset$ and $\eta$ is contained in $\overline{P}$.
Note that $e$ is freely homotopic to $\widetilde{e}$, so $\eta$ is also the geodesic closed curve homotopic to $e$ by uniqueness.
Since $\eta$ is in the homotopy class of $\widetilde{e}$, together with the fact that a geodesic curve cannot bound any disc,
it follows that $\overline{P}\backslash\eta$ is also a disjoint union of three cylinders.
\end{proof}

\begin{lem}\label{lem:e1e2}
Let $e_1$, $e_2$ be two figure eights (not necessarily from the same figure eight collection), and let $\eta_1$, $\eta_2$ be the unique geodesic figure eights freely homotopic to $e_1$ and $e_2$, respectively.
Then $\eta_1\cap\eta_2\ne\emptyset$ implies $e_1\cap e_2\ne\emptyset$.
\end{lem}

\begin{figure}[h]
\centering
\includegraphics{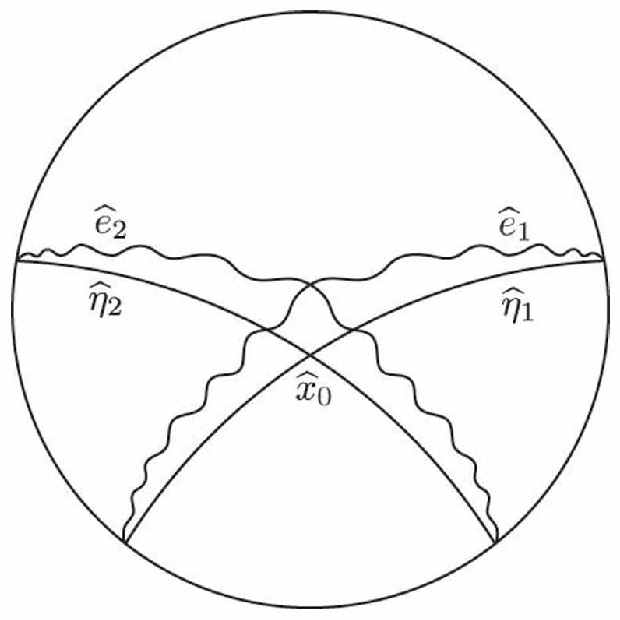}
\caption{The Poincar\'e disc $\PD$. }
\label{fig:Poincare}
\end{figure}
\begin{proof}
Two geodesic curves with non-empty intersection either have transversal intersection or coincide.
If $\eta_1$ and $\eta_2$ have transversal intersection, then the result follows immediately from Theorem \ref{thm:trans}, since $e_1$ and $e_2$ are in particular non-contractible closed curves.
Assume that $\eta_1$ and $\eta_2$ coincide.
Denote by $x_0$ the unique transversal double point of the figure eight $\eta_1$.
Denote by $\widehat{x}_0$ a lift of $x_0$ to the universal cover of the surface, modeled by the Poincar\'e disc $\PD$.
In a small neighborhood of $x_0$ there are two geodesic segments with transversal intersection.
Lift this neighborhood to a neighborhood of $\widehat{x}_0$ in the universal cover, and extend these two geodesics uniquely in $\PD$.
Denote them by $\widehat{\eta}_1$ and $\widehat{\eta}_2$.
By Theorem \ref{thm:prelim}, there exist lifts of $e_1$ and $e_2$, denoted by $\widehat{e}_1$ and $\widehat{e}_2$ respectively,
such that for $i=1,2$, $\widehat{\eta}_i$ and $\widehat{e}_i$ have the same endpoints at the circle at infinity (Figure \ref{fig:Poincare}).
Since $\widehat{\eta}_1$ and $\widehat{\eta}_2$ have transversal intersection, the endpoints of $\widehat{\eta}_1$ separate the endpoints of $\widehat{\eta}_2$.
As a result,  $\widehat{e}_1$ and $\widehat{e}_2$ must intersect in $\PD$, so their projection on the surface must also intersect as required.
\end{proof}

\begin{lem}\label{lem:coincide}
Let $P$ be a pair of pants, and let $e_1,e_2$ be two figure eights contained in $P$, such that for each $i$, $P\backslash e_i$
is a disjoint union of three cylinders.
Then $e_1\cap e_2\ne\emptyset$.
\end{lem}

\begin{proof}
Assume that $e_1\cap e_2=\emptyset$. Since $e_2$ is connected, it is contained in one of the connected components of $P\backslash e_1$.
By our assumption this component is a cylinder, denoted by $C$, so $e_2\subset C$.
The figure eight $e_2$ is composed of two disjoint simple loops with a unique intersection point.
But in a cylinder, each two non-contractible simple loops are freely homotopic to each other.
Hence, $e_2$ bounds a disc, in contradiction to the assumption that $P\backslash e_2$ is a disjoint union of cylinders.
Therefore $e_1\cap e_2\ne \emptyset$.
\end{proof}

Before the next proposition we wish to emphasize that we do not regard the boundary of a pair of pants as part of it,
i.e. $\partial P\cap P=\emptyset$.

\begin{prop}\label{prop:P1P2trans}
Let $P_1$,$P_2\subset\BM$ be two pairs of pants with geodesic boundaries such that $P_1\cap P_2\ne\emptyset$.
Then either the boundary components of $P_1$ and $P_2$ coincide,
or there exists a boundary component of $P_1$ that has transversal intersection with a boundary component of $P_2$.
\end{prop}

\begin{proof}
We will assume that not all the boundary components of $P_1$ coincide with those of $P_2$,
i.e. $(\partial P_1\backslash\partial P_2)\cup(\partial P_2\backslash\partial P_1)\ne\emptyset$.
We will first show that $(\partial P_1\cap P_2)\cup(\partial P_2\cap P_1)\ne\emptyset$.
Choose $x\in P_1\cap P_2$ and $y\in (\partial P_1\backslash\partial P_2)\cup(\partial P_2\backslash\partial P_1)$.
Assume without loss of generality that $y\in \partial P_1\backslash\partial P_2$.
If $y\in P_2$, then $y\in (\partial P_1\cap P_2)\cup(\partial P_2\cap P_1)$ as required.
Otherwise, i.e. $y\notin P_2$, choose a path $\gamma:[0,1]\rightarrow P_1\cup \partial P_1$ such that $\gamma(0)=x$, $\gamma([0,1))\subset P_1$,
and $\gamma(1)=y\in \partial P_1$.
By the above assumptions, $\gamma(1)\notin P_2\cup\partial P_2$ and $\gamma(0)=x\in P_1\cap P_2\subset P_2$.
Hence, there must exist $t_0\in(0,1)$ such that $\gamma(t_0)\in\partial P_2$.
But $\gamma([0,1))\subset P_1$, so $\gamma(t_0)\in\partial P_2\cap P_1\subset (\partial P_1\cap P_2)\cup(\partial P_2\cap P_1)$ as required.
\\
Now, choose $z\in (\partial P_1\cap P_2)\cup(\partial P_2\cap P_1)$ and assume without loss of generality that $z\in\partial P_1\cap P_2$.
Denote by $\delta$ the boundary component in $\partial P_1$ that contains $z$.
Note that $\delta$ does not coincide with any of the boundary components of $P_2$
since $\partial P_2\cap P_2=\emptyset$ and if $\delta\subset\partial P_2$ then $\delta\cap P_2=\emptyset$ but $z\in \delta\cap P_2$.
We will show that $\delta\cap\partial P_2\ne\emptyset$.
Suppose that $\delta\cap\partial P_2=\emptyset$, then since $z\in\delta\cap P_2$ we get that $\delta$ must be contained in $P_2$.
The boundary component $\delta$ is a simple closed curve,
and topologically, a pair of pants can be viewed as a sphere with three points removed,
so $\delta$ must bound a disc or a punctured disk.
Hence, $\delta$ is either contractible, or freely homotopic to one of the boundary components of $P_2$.
But $\delta$ is a geodesic, thereby it is not contractible,
and if it is homotopic to one of the geodesic boundary components of $P_2$ then they must coincide by Theorem \ref{thm:prelim},
in contradiction to the choice of $\delta$.
Therefore, $\delta\cap\partial P_2\ne\emptyset$.
If two geodesic curves intersect then they either coincide or have transversal intersection.
We have already shown that they do not coincide, so the result follows.

\end{proof}

\begin{figure}[t]
\centering
\includegraphics{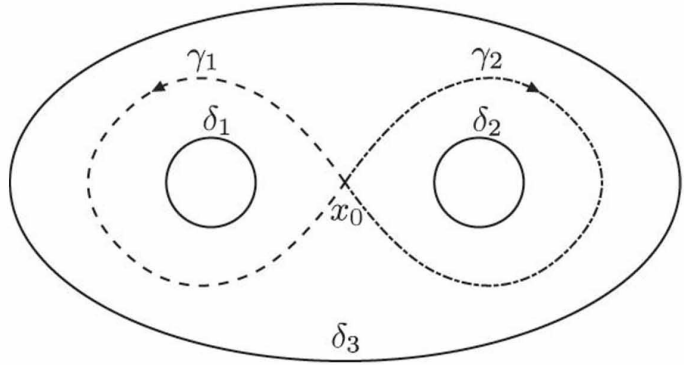}
\caption{$\gamma_1$ and $\gamma_2$ are freely homotopic to $\delta_1$ and $\delta_2$ respectively.
$\gamma_3:=\gamma_1*\gamma_2^{-1}$ is freely homotopic to $\delta_3$. }
\label{fig:Fig8}
\end{figure}

\begin{prop}\label{prop:P1P2}
Let $P_1,P_2\subset\BM$ be two pairs of pants with geodesic boundaries, such that $P_1\cap P_2\ne \emptyset$.
Let $e_1,e_2$ be two figure eights contained in $P_1$ and $P_2$, respectively, such that for $i=1,2$ $P_i\backslash e_i$ is a disjoint union of three cylinders.
Then $e_1\cap e_2\ne \emptyset$.
\end{prop}

\begin{proof}
We will first make the following observation.
Let $e$ be a figure eight contained in a pair of pants $P$ with geodesic boundaries such that $P\backslash e$ is a disjoint union of three cylinders.
Let $x_0$ be the unique transversal intersection point of the figure eight $e$.
The figure eight can be divided into two simple closed curves $\gamma_1$ and $\gamma_2$ with endpoints at $x_0$.
Define $\gamma_3$ to be the (non-smooth) curve $\gamma_1$ concatenated with $\gamma_2$ in reverse orientation.
Since $P\backslash e$ is a disjoint union of three cylinders,
we get that the curves $\gamma_1$,$\gamma_2$ and $\gamma_3$ are freely homotopic to the three boundary components of $P$ (Figure \ref{fig:Fig8}).
\par
Now, let $P_1$ and $P_2$ be two pairs of pants with geodesic boundaries such that $P_1\cap P_2\ne\emptyset$.
If $P_1$ and $P_2$ coincide then the result follows from Lemma \ref{lem:coincide}.
In the case that $P_1$ and $P_2$ do not coincide,
then by Proposition \ref{prop:P1P2trans} there exists a boundary component $\delta$ of $P_1$ that has transversal intersection with a boundary component $\delta'$ of $P_2$.
Let $\gamma_i$ and $\gamma_i'$, $i=1,2,3$ be the closed curves corresponding to the figure eights $e_1$ and $e_2$ respectively, as defined above.
Let $k,l\in\{1,2,3\}$ be such that the curves $\gamma_k$ and $\gamma_l'$ are freely homotopic to $\delta$ and $\delta'$ respectively.
By Theorem \ref{thm:trans} we get that $\gamma_k$ intersects $\gamma_l'$ and since $\gamma_k$ and $\gamma_l'$ are contained in $e_1$ and $e_2$ respectively,
we get that $e_1\cap e_2\ne\emptyset$ as required.
\end{proof}

We will need the following definitions in order to state Hall's marriage theorem.
\begin{defn}
Let $\mathcal{S}=\{S_1,...,S_n\}$ be a collection of finite subsets of some larger set $X$.
A \textdf{system of distinct representatives} is a set $R=\{r_1,...,r_n\}$ of pairwise distinct elements of $X$ with the property that for each $i=1,...,n$, $r_i\in S_i$.
\\
$\mathcal{S}$ satisfies the \textdf{marriage condition} if for any subset $\mathcal{T}=\{T_i\}$ of $\mathcal{S}$, $|\bigcup T_i|\ge |\mathcal{T}|$, i.e. any $k$ subsets taken together have at least $k$ elements.
\end{defn}

\begin{thm}\label{thm:Hall}
\textdf{Hall's marriage theorem} \cite{Hall}. Let $\mathcal{S}=\{S_1,...,S_n\}$ be a collection of finite subsets of some larger set.
Then there exists a system of distinct representatives of $\mathcal{S}$ if and only if $\mathcal{S}$ satisfies the marriage condition.
\end{thm}

\begin{thm}\label{thm:intersex}
Let $F,G:\BM\rightarrow\real$ be generic Morse functions.
Let $\{e_1,...,e_{2g-2}\}$ and $\{f_1,...,f_{2g-2}\}$ be the figure eight collections of $F$ and $G$ respectively.
Then there exists a permutation $\sigma\in S(2g-2)$ such that $e_i\cap f_{\sigma(i)}\ne\emptyset$ for each $i=1,...,2g-2$.
\end{thm}

\begin{proof}
Let $\{\overline{P}_1,...,\overline{P}_{2g-2}\}$ and $\{\overline{Q}_1,...,\overline{Q}_{2g-2}\}$ be the pair of pants decompositions of $\BM$ associated with $F$ and $G$ respectively.
We will first show that there exists a permutation $\sigma\in S(2g-2)$ such that $\overline{P}_i\cap \overline{Q}_{\sigma(i)}\ne\emptyset$.
We will use Hall's marriage theorem in order to prove this.
Define for $i=1,...,2g-2$
\[ S_i:=\{j|\overline{Q}_j\cap \overline{P}_i\ne\emptyset\} \]
so $S_i$ contains the indices of pairs of pants in $\{\overline{Q}_1,...,\overline{Q}_{2g-2}\}$ that intersect $\overline{P}_i$.
Define $\mathcal{S}:=\{S_1,...,S_{2g-2}\}$.
We will prove the marriage condition for $\mathcal{S}$.
Let $\mathcal{T}=\{S_i | i\in J\}$ be a subset of $\mathcal{S}$, where $J\subset\{1,...,2g-2\}$.
Define $\mathcal{P}:=\{\overline{P}_i | i\in J\}$.
Note that the hyperbolic area of any pair of pants with geodesic boundaries is equal by Gauss-Bonnet to $\frac{1}{2g-2}Vol(\BM)$.
It follows that the union of pairs of pants in $\mathcal{P}$ covers a total area of $\frac{|J|}{2g-2}Vol(\BM)$.
Thereby, this union must intersect at least $|J|$ pairs of pants in $\{\overline{Q}_1,...,\overline{Q}_{2g-2}\}$.
Equivalently, $|\bigcup_{i\in J}S_i|\ge|J|=|\mathcal{T}|$ and the marriage condition is proved.
Thus, there exists a system of distinct representatives $R=\{r_1,...,r_{2g-2}\}$ such that for $i=1,...,2g-2$, $r_i\in S_i$.
We can now define a permutation $\sigma\in S(2g-2)$ by $\sigma(i):=r_i$ and for each $i=1,...,2g-2$, $\overline{P}_i\cap \overline{Q}_{\sigma(i)}\ne\emptyset$ as required.
\par
Now it is left to prove that $e_i\cap f_{\sigma(i)}\ne\emptyset$.
By Lemma \ref{lem:contained} the unique closed geodesics homotopic to $e_i$ and $f_{\sigma(i)}$, denoted by $\overline{e}_i$ and $\overline{f}_{\sigma(i)}$,
are contained in $\overline{P}_i$ and $\overline{Q}_{\sigma(i)}$ respectively.
Furthermore, $\overline{P}_i\backslash\overline{e}_i$ and $\overline{Q}_{\sigma(i)}\backslash \overline{f}_{\sigma(i)}$, are both a disjoint union of three cylinders.
By Proposition \ref{prop:P1P2} we get that $\overline{e}_i\cap \overline{f}_{\sigma(i)}\ne\emptyset$ and from Lemma \ref{lem:e1e2} we conclude that $e_i\cap f_{\sigma(i)}\ne\emptyset$.
This completes the proof.
\end{proof}

\ftsection[\secsize]{Monotonicity}{sec:monoton}

\begin{thm} \label{thm:monotone}
Let $F,G:\BM\rightarrow\real$ be generic Morse functions,
such that $F\le G$.
Then $\xi(F)\le\xi(G)$.
\end{thm}

\begin{proof}
Let $\ECP_F=\{x_1,...,x_{2g-2}\}$ and $\ECP_G=\{y_1,...,y_{2g-2}\}$ be the sets of essential critical points of $F$ and $G$ respectively.
Using Definition \ref{def:QS} of $\xi$ and Theorem \ref{thm:pi2} we get that
\[ \xi(F)=\frac{1}{Vol(\BM)}\sum_{i=1}^{2g-2}F(x_i) \]
and
\[ \xi(G)=\frac{1}{Vol(\BM)}\sum_{i=1}^{2g-2}G(y_i) .\]
Let $\{e_1,...,e_{2g-2}\}$ and $\{f_1,...,f_{2g-2}\}$ be the figure eight collection of $F$ and $G$ respectively.
By Theorem \ref{thm:intersex} there exists a permutation $\sigma\in S(2g-2)$ such that for $i=1,...,2g-2$ we have $e_i\cap f_{\sigma(i)}\ne\emptyset$.
For each $i$, choose a point $z_i\in e_i\cap f_{\sigma(i)}$.
Then
\[F(x_i)= F(e_i)= F(z_i)\le G(z_i)= G(f_{\sigma(i)})= G(y_{\sigma(i)}),\]
which implies
\[ \xi(F)= \frac{1}{Vol(\BM)}\sum_{i=1}^{2g-2}F(x_i) \le \frac{1}{Vol(\BM)}\sum_{i=1}^{2g-2}G(y_{\sigma(i)})=\xi(G) \]
as required.

\end{proof}

\ftsection[\secsize]{Continuity}{sec:cont}
In this section we will examine the construction of Py's quasi-morphism as defined in \cite{Py} and show that it is continuous on time independent Hamiltonians,
with respect to the $C^2$-topology.
Lets recall the following definitions.
\\
\par
\begin{defn}
A \textdf{contact form} $\alpha$ on a $2n+1$ dimensional manifold $P$ is a 1-form with the property that $\alpha\wedge (d\alpha)^n\ne 0$.
\end{defn}

\begin{defn}
Given a contact form $\alpha$ on a manifold $P$,
the \textdf{Reeb vector field} $X$ is defined to be the unique vector field that satisfies $d\alpha(X,Z)=0$ for every $Z\in TP$ and $\alpha(X)=1$.
\end{defn}

\begin{defn}
A \textdf{principal G-bundle} is a fiber bundle $\pi:P\rightarrow \BM$ together with a smooth right action
$P \times G \rightarrow P$ by a Lie group G such that G preserves the fibers of P and acts freely and transitively on them.
The abstract fiber of the bundle is taken to be G itself.
\end{defn}

The following result is due to Banyaga \cite{Banyaga}.
\begin{thm}\label{thm:split}
Let $P$ be a closed connected manifold equipped with a contact form $\alpha$.
Let $\pi:P \rightarrow \BM$ be a principal $S^1$-bundle, such that the Reeb vector field on $P$ associated with
$\alpha$ coincides with the vector field generated by the action of $S^1$, parameterized by $\mathbb{R}/\mathbb{Z}$, on $P$.
Furthermore, we will assume that $\BM$ has a symplectic form $\omega$ that satisfies $\pi^*\omega=d\alpha$.
Then there exists a central extension by $S^1$ of the group $\mbox{Ham}(\BM,\omega)$,
\[ 0\rightarrow S^1 \rightarrow \mbox{Diff}_0(P,\alpha) \rightarrow \mbox{Ham}(\BM,\omega)\rightarrow 0 \]
where $\mbox{Diff}_0(P,\alpha)$ stands for the group of diffeomorphisms on $P$ which preserve $\alpha$ and are isotopic to the identity via an isotopy that preserves $\alpha$.
Moreover, when $\mbox{Ham}(\BM,\omega)$ is simply connected then the extension splits.
\end{thm}

In our case, $\BM$ is a closed surface of genus $g\ge 2$ hence $\mbox{Ham}(\BM,\omega)$ is simply connected [11,19] and the extension splits.
\\
\par
Let $\phi_{H_t}$ be a Hamiltonian diffeomorphism generated by the Hamiltonian $H_t$,
where $\int_{\BM}H_t\omega=0$ for every $t\in[0,1]$.
Define a vector field on $P$,
\[ V_t:= \widehat{sgradH_t}+(H_t\circ \pi)X,\]
where $X$ is the Reeb vector field on $P$,
and  $\widehat{sgradH_t}$ is the horizontal lift of $sgradH_t$, i.e. $\alpha(\widehat{sgradH_t})=0$ and $\pi_*(\widehat{sgradH_t})=sgradH_t$.
Define $\Theta(H_t)$ to be the flow generated by $V_t$.
It can be shown that $V_t$ preserves $\alpha$ and that the homotopy class with fixed endpoints of $\Theta(H_t)$ depends only on the homotopy class with fixed endpoints of the flow generated by $H_t$.
Hence, $\Theta:\widetilde{\mbox{Ham}}(\BM,\omega)\rightarrow \widetilde{\mbox{Diff}}_0(P,\alpha)$
(where $\widetilde{G}$ stands for the universal cover of $G$) is well defined.
Since $\mbox{Ham}(\BM,\omega)$ is simply connected then $\Theta$ can be defined on $\mbox{Ham}(\BM,\omega)$,
and by taking the time one map of $\Theta(\phi_{H_t})$
we obtain the splitting map from $\mbox{Ham}(\BM,\omega)$ to $\mbox{Diff}_0(P,\alpha)$ of the extension in Theorem \ref{thm:split}.
\\
\par
Let $H\in C^\infty(\BM)$, i.e. $H$ is a time independent Hamiltonian.
In order to apply $\Theta$ on the flow generated by $H$, we must first normalize $H$, i.e. $H\mapsto H-\frac{1}{Vol(M)}\int_{\BM}H\omega$.
The normalization mapping is obviously smooth.
The definition of the vector field $V_t$ involves $sgradH$,
hence $V_t$ is continuous as a function of $H$ with respect to the $C^2$-topology on $C^\infty(\BM)$, and so is the flow $(\Theta(H))_t$.
\\
\\
\textdf{Construction of the quasi-morphism.}
Let $\BM$ be a closed surface of genus $g\ge 2$, equipped with a symplectic form $\omega$.
We will assume that the total area of $\BM$ is equal to $2g-2$.
Choose a metric with constant negative curvature on $\BM$ such that its associated area form is equal to $\omega$.
Denote by $P$ the unit tangent bundle of $\BM$.
We will use the Poincar\'e disc $\PD$ as a model for the universal cover of $\BM$,
and denote by $S^1\PD$ the unit tangent bundle of $\PD$.
We can define a $S^1$-principal fiber bundle on $P$ and $S^1\PD$ by rotating each vector in the unit tangent bundle by the same angle as defined by the metric.
We will write $S^1_{\infty}$ for the circle at infinity of $\PD$ and $p_{\infty}:S^1\PD\rightarrow S^1_{\infty}$ for the natural projection,
sending each unit vector in the tangent bundle of $\PD$  to the limit at $S^1_{\infty}$ of the unique geodesic tangent to it.
Note that $p_{\infty}$ is a smooth mapping.
We will denote by $\pi:P\rightarrow\BM$ the natural projection.
Denote by $X$ the vector field on $P$ generated by the action of $S^1$, parameterized by $\mathbb{R}/\mathbb{Z}$, on $P$.
One can show that there exists a contact form $\alpha$ on $P$ such that $\pi^*\omega=d\alpha$ and its Reeb vector field coincides with $X$.
Hence, according to Theorem \ref{thm:split} we can construct the homomorphism $\Theta$ as defined above.
Given an Hamiltonian $H$ on $\BM$, we can define an isotopy $(\Theta(H))_t$ on $P$ as constructed above.
Note that since $P$ is closed, $(\Theta(H))_t$ is uniformly continuous on $P$.
Let $\widehat{(\Theta(H))_t}:S^1\PD\rightarrow\ S^1\PD$ be a lift of $(\Theta(H))_t$ from $P$ to $S^1\PD$.
Thus, for every $v\in S^1\PD$ we can define a curve in $S^1$, $\gamma^{(H,v)}:[0,1]\rightarrow S^1$, by
\[ \gamma^{(H,v)}(t):=p_{\infty}(\widehat{(\Theta(H))_t}(v)) .\]
Parameterize $S^1$ by $\real / {\mathbb Z}$ and let $\widetilde{\gamma^{(H,v)}}$ be a lift to $\real$ of $\gamma^{(H,v)}$.
Define
\[ Rot(H,v):= \widetilde{\gamma^{(H,v)}}(1) - \widetilde{\gamma^{(H,v)}}(0) .\]
Note that $(\Theta(-))_t$ is continuous with respect to the $C^2$-topology, hence so is $Rot(-,v)$ as a composition of continuous maps.
Thereby, we can find $\delta>0$, such that if $H'\in C^{\infty}$ and $\| H-H'\|_{C^2}<\delta$, then for every $v\in S^1\PD$,
\[ |Rot(H,v)-Rot(H',v)| < 1 .\]
Denote by $\widetilde{\pi}$ the projection from $S^1\PD$ to $\PD$.
Now, define for every $\widetilde{x}\in\PD$
\[ \widetilde{angle}(H,\widetilde{x})=-\inf_{\widetilde{\pi}(v)=\widetilde{x}}\lfloor Rot(H,v)\rfloor, \]
where $\lfloor x\rfloor$ is the integer part of $x$.
It is shown in \cite{Py} that if $\widetilde{\pi}(v)=\widetilde{\pi}(w)$ then
\[ |\lfloor Rot(H,v)\rfloor - \lfloor Rot(H,w)\rfloor|\le 2 .\]
Hence, for $v\in S^1\PD$ such that $\widetilde{\pi}(v)=\widetilde{x}$
\[ |\widetilde{angle}(H,\widetilde{x}) - ( -\lfloor Rot(H,v)\rfloor)| \le 2 .\]
Obviously $|\lfloor x\rfloor - x|\le 1$, so altogether we obtain that
\[ |\widetilde{angle}(H,\widetilde{x})-\widetilde{angle}(H',\widetilde{x})| \le \]
\[ |\widetilde{angle}(H,\widetilde{x}) - (-Rot(H,v))|
 +  |Rot(H,v) - Rot(H',v)| + \]
\[ +  |(-Rot(H',v))- \widetilde{angle}(H',\widetilde{x})| \]
\[
\le 3+1+3 = 7.
\]

The function $\widetilde{angle}(H,-)$ is invariant by the action of the fundamental group of $\BM$,
so we can define a measurable bounded function $angle(H,-)$ on $\BM$.
Define
\[ \mu_1(\phi_H):=\int_{\BM}angle(H,-)\omega .\]
Note that
\[ |\mu_1(H)-\mu_1(H')| \le \int_{\BM}|angle(H,-)-angle(H',-)|\omega \le 7\cdot Vol(\BM). \]

We will sum up the result.

\begin{prop} \label{prop:k}
There exists a constant $K$ ($=7\cdot Vol(\BM)$) such that for $H\in C^{\infty}(\BM)$ there exists $\delta>0$
such that for any $H'\in C^{\infty}(M)$ that satisfies $\|H-H'\|_{C^2}<\delta$ we have
\[ |\mu_1(\phi_H)-\mu_1(\phi_{H'})|\le K .\]
\end{prop}

\begin{defn}
For $m\in {\mathbb N}$ and $\phi_H\in \mbox{Ham}(\BM,\omega)$ define
\[ \mu_m(\phi_H):=\frac{1}{m}\mu_1(\phi_H^m) .\]
\end{defn}

\begin{prop}\label{prop:qmbound}
For $H\in C^\infty(\BM)$ and $m\in {\mathbb{N}}$, there exists $\delta>0$ such that for $H'\in C^{\infty}(M)$ that satisfies $\|H-H'\|_{C^2}<\delta$
we have
\[ |\mu_m(\phi_H)-\mu_m(\phi_{H'})|\le \frac{K}{m} .\]
\end{prop}
\begin{proof}
Let $H\in C^\infty(\BM)$ and $m\in {\mathbb N}$. Using Proposition \ref{prop:k} for the function $mH$, there exists $\delta'>0$ such that
for every $G\in C^{\infty}(M)$ that satisfies $\|G-mH\|_{C^2}<\delta'$ we have  $|\mu_1(\phi_G)-\mu_1(\phi_{mH})|\le K$.
Choose $\delta:=\frac{\delta'}{m}$ so for $H'\in C^{\infty}(M)$ such that $\|H'-H\|_{C^2}<\delta$ we have $\|mH'-mH\|_{C^2}<\delta'$
which implies $|\mu_1(\phi_{mH'})-\mu_1(\phi_{mH})|\le K$.
With the observation that $\phi_{mH}=\phi_H^m$ we have $|\mu_1(\phi_{H'}^m)-\mu_1(\phi_{H}^m)|\le K$.
Dividing by $m$ we obtain
\[ |\mu_m(\phi_{H'})-\mu_m(\phi_H)|<\frac{K}{m},\]
as required.
\end{proof}

\begin{prop} \label{prop:qmuniform}
Let $\mu_1:G\rightarrow\real$ be a quasi-morphism on a group $G$ with defect $C>0$.
Define for $m\in {\mathbb N}$ and $x\in G$, $\mu_m(x):=\frac{1}{m}\mu_1(x^m)$.
Let $\mu_{\infty}(x):=\lim_{n\rightarrow\infty}\mu_m(x)$ be the homogenization of $\mu_1$.
Then for every $x\in G$ and $m\in {\mathbb N}$
\[ |\mu_{\infty}(x)-\mu_m(x)|\le \frac{C}{m} .\]
\end{prop}
\begin{proof}
Let $x\in G$, $m,p\in {\mathbb N}$. By the quasi-morphism property we get
\[ |\mu_1(x^{mp})-\mu_1(x^m)-\mu_1(x^{m(p-1)})|<C .\]
Using induction on $p$ we obtain
\[ |\mu_1(x^{mp})-p\cdot \mu_1(x^m)|<pC .\]
Divide by $mp$
\[ | \frac{\mu_1(x^{mp})}{mp} - \frac{\mu_1(x^m)}{m}|< \frac{C}{m} .\]
Equivalently,
\[ |\mu_{mp}(x)-\mu_m(x)|< \frac{C}{m} .\]
As $p$ tends to infinity we get
\[ |\mu_{\infty}(x)-\mu_m(x)|\le \frac{C}{m} ,\]
as required.
\end{proof}

\begin{thm} \label{thm:cont}
The functional $\xi:C^{\infty}(\BM)\rightarrow\real$ (see Definition \ref{def:QS}) is continuous with respect to the $C^2$-topology.
\end{thm}
\begin{proof}
Let $H\in C^{\infty}(\BM)$ and $\epsilon>0$.
Let $C>0$ be the defect of $\mu_1$ and $K>0$ the constant defined in Proposition \ref{prop:k}.
Choose $N\in {\mathbb N}$ such that $N>max(\frac{4C}{\epsilon},\frac{2K}{\epsilon})$.
By Proposition \ref{prop:qmbound} there exists $\delta>0$ such that for $H'$ that satisfies $\|H-H'\|_{C^{\infty}}<\delta$ we get
\[ |\mu_N(\phi_H)-\mu_N(\phi_{H'})|<\frac{K}{N} .\]
According to Proposition \ref{prop:qmuniform} we have
\[ |\mu_\infty(\phi_H)-\mu_N(\phi_H)|<\frac{C}{N} \]
and
\[ |\mu_\infty(\phi_{H'})-\mu_N(\phi_{H'})|<\frac{C}{N} .\]
Thus
\[ | \mu_\infty(\phi_H)-\mu_\infty(\phi_{H'})|\le \]
\[ |\mu_\infty(\phi_H)-\mu_N(\phi_H)| + |\mu_N(\phi_H)-\mu_N(\phi_{H'})| + |\mu_\infty(\phi_{H'})-\mu_N(\phi_{H'})| < \]
\[< \frac{2C}{N}+\frac{K}{N} < \epsilon .\]

Py's quasi-morphism $\mu$ is defined to be $\mu_\infty$, so the result follows from Definition \ref{def:QS} of $\xi.$
\end{proof}

\textdf{Acknowledgements.}
I would like to express my sincere thanks to my thesis advisor,
Professor Leonid Polterovich, for his dedicated and patient guiding, and for the time he spent sharing with me his knowledge and expertise.
I would like to thank the Israel Science Foundation grant \# 11/03 which partially supported this work.

\end{document}